\documentclass[12pt,letterpaper,titlepage]{amsart}
\usepackage{amsmath, amssymb, amsthm, amsfonts,amscd,amsaddr}
\usepackage[usenames]{color}

\theoremstyle{plain}
\newtheorem{theorem}{Theorem}[section]

\newtheorem{cor}[theorem]{Corollary}

\newtheorem{prop}[theorem]{Proposition}
\newtheorem{lemma}[theorem]{Lemma}
\theoremstyle{definition}
\newtheorem{definition}[theorem]{Definition}

\newtheorem{rmk}[theorem]{Remark}
\numberwithin{equation}{section}
\newtheorem*{theoremA*}{Theorem A}
\newtheorem*{theoremB*}{Theorem B}
\newtheorem*{theoremm1*}{Theorem A'}
\newtheorem*{theoremC*}{Theorem C}
\newtheorem*{theoremD*}{Theorem D}
\newtheorem*{theoremE*}{Theorem E}
\newtheorem*{theoremF*}{Theorem F}
\newtheorem*{theoremE2*}{Theorem E2}
\newtheorem*{theoremE3*}{Theorem E3}

\newcommand{\bs}{\backslash}

\newcommand{\C}{\mathbb{C}}

\newcommand{\M}{\mathcal{M}}

\newcommand{\Q}{\mathbb{Q}}
\newcommand{\Z}{\mathbb{Z}}

\newcommand{\R}{\mathbb{R}}

\newcommand{\rk}{\operatorname{rk}}

\newcommand{\GL}{\operatorname{GL}}

\newcommand{\im}{\operatorname{im}}

\newcommand{\Ad}{\operatorname{Ad}}
\newcommand{\ad}{\operatorname{ad}}

\def\hat{\widehat}
\def\af{\mathfrak{a}}
\def\bmf{\mathfrak{b}}
\def\cf{\mathfrak{c}}
\def\df{\mathfrak{d}}

\def\gf{\mathfrak{g}}

\def\hf{\mathfrak{h}}
\def\kf{\mathfrak{k}}
\def\lf{\mathfrak{l}}
\def\mf{\mathfrak{m}}
\def\nf{\mathfrak{n}}

\def\pf{\mathfrak{p}}
\def\qf{\mathfrak{q}}

\def\sf{\mathfrak{s}}
\def\sl{\mathfrak{sl}}

\def\uf{\mathfrak{u}}

\def\zf{\mathfrak{z}}
\def\la{\langle}
\def\ra{\rangle}
\def\1{{\bf1}}

\def\M{\mathcal{M}}
\def\oline{\overline}

\makeatletter
\tagsleft@false
\makeatother

\title[Real spherical varieties] {The local structure theorem for real
  spherical varieties}

\subjclass[2010]{14L30, 14M17, 14M27, 22F30} \keywords{spherical
  varieties, homogeneous spaces, real reductive groups}

\begin{document}
\date{December 23, 2014}

\begin{abstract}
  Let $G$ be an algebraic real reductive group and $Z$ a real
  spherical $G$-variety, that is, it admits an open orbit for a
  minimal parabolic subgroup $P$.  xWe prove a local structure theorem
  for $Z$. In the simplest case where $Z$ is homogeneous, the theorem
  provides an isomorphism of the open $P$-orbit with a bundle
  $Q\times_L S$. Here $Q$ is a parabolic subgroup with Levi
  decomposition $L\ltimes U$, and $S$ is a homogeneous space for a
  quotient $D=L/L_{\mathrm n}$ of $L$, where $L_{\mathrm n}\subseteq
  L$ is normal, such that $D$ is compact modulo center.
\end{abstract}

\author[Knop]{Friedrich Knop} \email{friedrich.knop@fau.de}
\address{FAU Erlangen-N\"urnberg, Department Mathematik,\\Cauerstra\ss
  e 11, D-91058 Erlangen, Germany} \author[Kr\"otz]{Bernhard
  Kr\"{o}tz} \email{bkroetz@math.uni-paderborn.de}
\address{Universit\"at Paderborn, Institut f\"ur
  Mathematik,\\Warburger Stra\ss e 100, D-33098 Paderborn, Germany}
\thanks{The second author was supported by ERC Advanced Investigators
  Grant HARG 268105} \author[Schlichtkrull]{Henrik Schlichtkrull}
\email{schlicht@math.ku.dk} \address{University of Copenhagen,
  Department of Mathematics,\\Universitetsparken 5, DK-2100 Copenhagen
  \O, Denmark\\ \ \\ \ }
\maketitle

\section{Introduction}

Let $G_\C$ be a complex reductive group and $B_\C < G_\C$ a fixed
Borel subgroup.  We recall that a normal $G_\C$-variety $Z_\C$ is
called {\it spherical} provided that $B_\C$ admits an open orbit. The
local nature of a spherical variety is given in terms of the local
structure theorem, \cite{BLV}, \cite{Kn}.  In its simplest form,
namely applied to a homogeneous space $Z_\C= G_\C/H_\C$ for which
$B_\C H_\C$ is open, it asserts that there is a parabolic subgroup
$Q_\C>B_\C$ with Levi-decomposition $Q_\C = L_\C \ltimes U_\C$ such
that
the action of $Q_\C$ on $Z_\C$ induces an isomorphism of
$(L_\C/L_\C\cap H_\C)\times U_\C$ onto $B_\C H_\C$.

\par The purpose of this paper is to continue the geometric study of
real spherical varieties begun in \cite{KS}.  We let $G$ be an
algebraic real reductive group and $Z$ a normal real algebraic
$G$-variety. Then $Z$ is called {\it real spherical} provided a
minimal parabolic subgroup $P<G$ has at least one open orbit on $Z$.
With this assumption on $Z$ we prove a local structure theorem
analogous to the one above.  In particular, when applied to a
homogeneous real spherical space $Z=G/H$ with $PH$ open, it yields a
parabolic subgroup $Q> P$ with Levi-decomposition $Q=L \ltimes U$ such
that
\begin{equation*} \label{LST2} L_{\mathrm n} < Q \cap H <
  L\,.  \end{equation*} Here $L_{\mathrm n} \triangleleft L$ denotes
the product of all non-compact non-abelian normal factors of $L$.
Furthermore, the action of $Q$ induces a diffeomorphism of $(L/L\cap
H)\times U$ onto $PH$.

\par Our proof of the real local structure theorem is based on the
symplectic approach of \cite{Kn}. Our investigations also show the
number of $G$-orbits on a real spherical variety is finite. Combined
with the main result of \cite{KS} it implies that the number of
$P$-orbits on a real spherical variety is finite.

\bigskip {\bf Acknowledgments}: The authors thank the referees for
suggesting substantial improvements to the paper.

\section{Homogeneous spherical spaces}

Lie groups in this paper will be denoted by upper case Latin letters,
$A$, $B$ \ldots, and their associated Lie algebras with the
corresponding lower case Gothic letter $\af$, $\bmf$ etc.

\par For a Lie group $G$ we denote by $G_0$ its connected component
containing the identity, by $Z(G)$ the center of $G$ and by $[G,G]$
the commutator subgroup.

\par On a real reductive Lie algebra $\gf$ we fix a non-degenerate
invariant bilinear form $\kappa(\cdot, \cdot)$, for example the
Cartan-Killing form in case $\gf$ is semisimple.

A Lie group $G$ will be called {\it real reductive} provided that
\begin{itemize}
\item The Lie algebra $\gf$ is reductive.
\item There exists a maximal compact subgroup $K<G$ such that we have
  a homeomorphism (polar decomposition)
$$ K \times \sf \to G, \ \ (k,X)\mapsto k\exp(X)$$
where $\sf:= \kf^{\perp_\kappa}$
\end{itemize}

Observe that for a real reductive group the bilinear form $\kappa$ can
(and will) be chosen $K$-invariant. A real reductive group is called
\emph{algebraic} if it is isomorphic to an open subgroup of the group
of real points $G_\C(\R)$ where $G_\C$ is a reductive algebraic group
which is defined over $\R$.

Let now $G$ be a real reductive group, and let $P$ be a minimal
parabolic subgroup. The unipotent part of $P$ is denoted $N$. If a
maximal compact subgroup $K$ as above has been chosen, with associated
Cartan involution $\theta$ of $G$, a maximal abelian subspace
$\af\subset\sf$ can also be chosen.  These choices then induce an
Iwasawa decomposition $G=KAN$ of $G$ and a Langlands decomposition
$P=MAN$ of $P$. Here $M=Z_K(\af)$.  However, at present we do not fix
$K$ and $\af$.

\par Let $H$ be a closed subgroup of $G$ such that $H/H_0$ is finite.
Recall that $Z=G/H$ is said to be {\it real spherical}, if the minimal
parabolic subgroup $P$ admits an open orbit on $Z$. Furthermore, in
this case $H$ is called a {\it spherical subgroup}.  Note that $H$ is
not necessarily reductive.

\begin{rmk} Here a remark on terminology is in order. Historically,
  spherical subgroups were first introduced by M.~Kr\"amer in the
  context of compact Lie groups, see \cite{Kraemer}. However, as our
  focus is to investigate non-compact homogeneous spaces we allow a
  discrepancy between the original definition and the current one. In
  fact with our definition every closed subgroup of $G$ is spherical
  if $G$ is compact.
\end{rmk}

We denote by $z_0\in Z$ the origin of the homogeneous space $Z=G/H$.

\subsection{Semi-invariant functions and the local structure theorem}

Let $G$ be a real reductive Lie group.

\begin{definition} \label{defi semispherical} Let $Z=G/H$ with
  $H\subseteq G$ a closed subgroup.

  (1) A finite dimensional real representation $(\pi, V)$ of $G$ is
  called {\it $H$-semispherical} provided there is a cyclic vector
  $v_H \in V$ and a character $\gamma: H \to \R^\times$ such that
$$\pi(h)v_H = \gamma(h) v_H,\qquad \forall h\in H.$$  

(2) The homogeneous space $Z$ is called {\it almost algebraic} if
there exists an $H$-semispherical representation $(\pi, V)$ such that
the map
$$Z \to \mathbb{P}(V), \ \ g\cdot z_0 \mapsto [\pi(g)v_H]$$
is injective.
\end{definition}

According to a theorem of Chevalley (see \cite{Borel} Thm.~5.1),
$Z=G/H$ is almost algebraic if $G$ and $H$ are algebraic.  In the
following we always assume that $Z=G/H$ is almost algebraic.

For a reductive Lie algebra $\gf$ we write $\gf_{\mathrm n}$ for the
direct sum of the non-compact non-abelian ideals in $[\gf,\gf]$. If
$\gf$ is the Lie algebra of $G$, then $G_{\mathrm n}$ denotes the
corresponding connected normal subgroup of $[G,G]$.

\begin{theorem}[Local structure theorem, homogeneous case] \label{LST}
  Let $Z=G/H$ be an almost algebraic real spherical space, and let
  $P\subseteq G$ be a minimal parabolic subgroup such that $PH$ is
  open.  Then there is a parabolic subgroup $Q\supseteq P$ with
  Levi-decomposition $Q= LU $ such that:
  \begin{enumerate}
  \item The map
$$ Q \times_L (L/L\cap H) \to Z, \ \ [q,l (L\cap H)]\mapsto ql\cdot z_0$$
is a $Q$-equivariant diffeomorphism onto $Q\cdot z_0\subseteq Z$.
\item $Q\cap H\subseteq L$.
\item $L_{\mathrm n}\subseteq H$.
\item $(L\cap P) (L\cap H)=L$.
\item $QH=PH$.
\end{enumerate}
\end{theorem}

\begin{proof} The proof consists of an iterative procedure, in which
  we construct a strictly decreasing sequence of parabolic subgroups
$$Q_0\supset Q_1\supset \dots \supset P$$ 
and corresponding Levi subgroups $L_0\supset L_1\supset \dots,$ all
satisfying (1).  Note that (2) is an immediate consequence of (1).
After a finite number of steps a parabolic subgroup is reached which
also satisfies (3)--(5).

Let $Q_0=G$. It clearly satisfies (1).  If $G_{\mathrm n}\subseteq H$
then $PH=G$ since $P$ contains both the center of $G$ and every
compact normal subgroup of $[G,G]$.  Hence in this case $Q=Q_0$ solves
(1)--(5).  Note also that since $L\cap P$ is a minimal parabolic
subgroup of $L$, the argument just given, but applied to $L$, shows
that (4) and (5) are consequences of (3).

Assume now that $G_{\mathrm n}\not\subseteq H$.  By our general
assumption on $Z$ there is a finite dimensional representation
$(\pi,V)$ of $G$ and a vector $v_H\in V$ satisfying all the properties
of Definition \ref{defi semispherical}.  The assumption on $G_{\mathrm
  n}$ implies that $\pi(g)v_H\notin \R v_H$ for some $g\in G_{\mathrm
  n}$, hence $\pi$ does not restrict to a multiple of the trivial
representation of $G_{\mathrm n}$.

Choose a Cartan involution for $G$ and a maximal abelian subspace
$\af\subset\sf$, but note that these choices may be valid only for the
current step of the iteration.  Let $v^*\in V^*\setminus\{0\}$ be an
extremal weight vector so that the line $\R v^*$ is fixed by $AN$, say
$\pi^*(g)v^*=\chi(g)v^*$ for $g\in AN$ and some character $\chi: AN\to
\R^\times$. Now we need the following

\begin{lemma}\label{extra}
  Let $G$ be a connected semisimple Lie group without compact factors,
  and with minimal parabolic $P=MAN\subseteq G$. Let $V$ be a
  non-trivial finite dimensional irreducible real representation of
  $G$. Then $V^{AN}=\{0\}$.
\end{lemma}

\begin{proof}
  Let $\bar N=\theta(N)$ be the unipotent part of the parabolic
  subgroup $\theta(P)$ opposite to $P$.  It follows from the
  representation theory of $\sl(2,\R)$ that vectors in $V^{AN}$ are
  also fixed by $\bar N$.  Since $G$ has no compact factors it is
  generated by $\bar N$ and $AN$, hence $V^{AN}=V^G=\{0\}$.
\end{proof}

By this lemma and what was just seen, we can choose $v^*$ such that
$\chi$ is nontrivial on $G_{\mathrm n}\cap A$.  The matrix coefficient
$$f(g) := v^*(\pi(g) v_H)\qquad (g\in G)$$
satisfies $f(angh)= \chi(a)^{-1} \gamma(h)f(g)$ for all $g\in G$,
$an\in AN$ and $h\in H$.  As $v_H$ is cyclic and $v^*$ non-zero, and
as $PH$ is open, $f$ is not identically zero on $M$.

\par We construct a new function:
$$F(g):= \int_M  f(mg)^2 \ dm \qquad (g\in G) \, .$$
This function is smooth, $G$-finite, non-negative valued, and
satisfies
\begin{equation}\label{F(mangh)}
  F(man gh) =\chi(a)^{-2} \gamma^2(h)F(g)
\end{equation}
for all $g\in G$, $man\in P$ and $h\in H$.  Furthermore, $F(e)> 0$.

\par 
It follows from the $G$-finiteness together with (\ref{F(mangh)}) that
$F$ is a matrix coefficient
$$ F(g) = w^*(\rho(g) w_{H}) $$
of a finite dimensional representation $(\rho, W)$ of $G$, with
non-zero vectors $w_{H}\in W$ and $w^*\in W^*$ such that
$$\rho(h)w_H=\gamma(h)^2w_H, \quad \rho^*(man)w^*=\chi(a)^2w^*$$
for all $h\in H$ and $man\in P=MAN$. Here $W^*$ can be chosen to be
the span of all left translates of $F$. Since $F$ is a highest weight
vector, $W^*$ and hence $W$ is irreducible. Define $\nu\in\af^*$ by
$$e^{\nu(X)}=\chi(\exp X)^2,$$ 
then $\nu$ is the highest $\af$-weight of $\rho^*$, and it is dominant
with respect to the set $\Sigma(\af,\nf)$ of $\af$-roots in $\nf$.

\par Now define a subgroup $Q_1=Q\subseteq G$ to be the stabilizer of
$\R w^*$,
$$Q=\{g\in G\mid \rho^*(g)w^*\in\R w^*\},$$ 
and define a character $\psi: Q\to \R^\times$ by
$$\rho^*(g)w^*=\psi(g)w^*.$$ 
In particular, we see that $Q$ is a parabolic subgroup that contains
$P$. Moreover $\psi: Q\to\R$ extends $\chi^{2}:AN\to \R^+$.  Let
$U\subseteq Q$ be the unipotent radical of $Q$, its Lie algebra is
spanned by the root spaces of the roots $\alpha\in\Sigma(\af,\nf)$ for
which $\langle\alpha,\nu\rangle >0$.

Note that since $w_{H}$ is cyclic, $\rho^*(g)w^* = c w^*$ if and only
if $F(g^{-1}x)=cF(x)$ for all $x\in G$.  Hence
$$Q=\{g\in G\mid F(g\,\cdot) \text{ is a multiple of } F\}$$ 
and $F(q\,\cdot)=\psi(q)F$ for all $q\in Q$.  (We use the symbol
$F(g\,\cdot)$ for the function $x\mapsto F(gx)$ on $G$.)

We note that $Q\cap G_{\mathrm n}$ is a proper subgroup of $G_{\mathrm
  n}$, for otherwise $\rho^*$ would be one-dimensional spanned by
$w^*$, and this would contradict the non-triviality of its highest
weight $e^\nu=\chi^2$ on $G_{\mathrm n}\cap A$.

Set $Z_0:=QH\subseteq Z$, then $Z_0$ is open since $qPH$ is open for
each $q\in Q$.  Following \cite{Kn}, Th. 2.3, we define a moment-type
map:
$$\mu: Z_0 \to \gf^*, \ \ \mu(z)(X):= \frac {dF(q) (X)}{F(q)} = \frac{d}{dt}\Big|_{t=0} 
\frac{F(\exp(tX) q)}{F(q)}$$ for $q\in Q$ such that $z=qH \in Z_0$ and
$ X\in \gf$.  Note that this map is well-defined: $F(q)\neq 0$ for
$q\in Q$, and if $q\cdot z_0 = q'\cdot z_0$ then $q=q'h$ for some
$h\in H$.

We let $G$ act on $\gf^*$ via the co-adjoint action and record:

\begin{lemma} The map $\mu$ is $Q$-equivariant.
\end{lemma}

\begin{proof} Let $z\in Z_0$, $q\in Q$ and $Y\in \gf$. Then
  \begin{align*} \mu(qz)(Y) & = \frac {\frac{d}{dt}\Big|_{t=0} F
      (\exp(tY) qz)}{ F(qz)}=
    \frac {\frac{d}{dt}\Big|_{t=0} F (q q^{-1}\exp(tY) qz)}{\psi(q) F(z)}\\
    &= \frac{\frac{d}{dt}\Big|_{t=0} F (\exp(t\Ad(q^{-1})Y) z)}{F(z)}
    = (\Ad^*(q) \mu(z))(Y)\, .\tag*{\qedhere}
  \end{align*}
\end{proof}

Note that
\begin{equation} \label{mupsi} \qquad \mu(z)(X)= d\psi(X),\quad (X\in
  \qf)
\end{equation}
for all $z\in Z_0$. In particular,
$\mu(z_1)-\mu(z_2)\in\qf^\perp\subseteq\gf^*$ for $z_1,z_2\in
Z_0$. Moreover, $\mu(z)(X+Y)=-\nu(X)$ for $X\in\af$ and $Y\in
\mf+\nf$.

We now identify $\gf^*$ with $\gf$ via the invariant non-degenerate
form $\kappa(\cdot, \cdot)$, then $\qf^\perp$ is identified with
$\uf$, and $(\mf+\nf)^\perp$ with $\af+\nf$.
Let $$X_0=\mu(z_0)\in\af+\nf,$$ then $X_0\notin\nf$ since $\nu\neq 0$
and hence $X_0$ is a semisimple element.  Write $X_{s}$ for the
$\af$-part of $X_0$, then the eigenvalues of $\ad(X_0)$ on $\nf$ are
the $\alpha(X_{s})$ where $\alpha\in\Sigma(\af,\nf)$. By the
identification of $\gf^*$ with $\gf$ these are the inner products
$\langle -\nu,\alpha\rangle$, in particular they are all $\leq 0$ and
on $\uf$ they are $<0$.

We conclude from the above that $\im \mu \subseteq X_0 +\uf.$ We claim
equality:
\begin{equation} \label{adeq} \im\mu=X_0 + \uf \, .\end{equation} As
$\mu$ is $Q$-equivariant we have $\im \mu =\Ad(Q)X_0.$ The lemma below
(with $X=-\ad(X_0)$) implies $\Ad(U)X_0= X_0 + \uf$, and then
(\ref{adeq}) follows.

\begin{lemma} Let $\uf$ be a nilpotent Lie algebra and $X:\uf \to \uf$
  a derivation which is diagonalizable with non-negative eigenvalues.
  Then in the solvable Lie algebra $\gf:=\R X \ltimes \uf$ the
  following identity holds:
  \begin{equation}\label{e^ad}
    e^{\ad \uf} X  = X + [X,\uf]\, .
  \end{equation}
\end{lemma}

\begin{proof} 
  Note that $[X,\uf]=\uf$ if all eigenvalues are positive.  The
  inclusion $\subseteq$ in (\ref{e^ad}) is easy.  The proof of the
  opposite inclusion is by induction on $\dim\uf$, and the case
  $\dim\uf=0$ is trivial. Assume $\dim\uf>0$ and let
  $\uf=\sum_{\lambda\ge 0} \uf (X, \lambda)$ be the eigenspace
  decomposition of the operator $X: \uf \to \uf$. Let $\lambda_1\ge 0$
  be the smallest eigenvalue and set $\uf_1:=\uf (X, \lambda_1)$ and
  $\uf_2:=\sum_{\lambda>\lambda_1} \uf (X, \lambda)$.  Note that
  $\uf_2$ is an ideal in $\uf$, and $\uf=\uf_1 + \uf_2$ as vector
  spaces.
  \par
  By induction we have $e^{\ad \uf_2}X= X + \uf_2$.  If $\lambda_1=0$
  then $[X,\uf]=\uf_2$, and we are done.  Otherwise
  $[\uf_1,\uf_1]\subseteq\uf_2$ and hence
$$e^{\ad U}X\in X+\lambda_1 U+\uf_2.$$
for $U\in\uf_1$.  Note that $e^{\ad \uf}$ is a group as $\uf$ is
nilpotent. It follows that
\begin{align*} e^{\ad \uf}X & \supseteq e^{\ad \uf_1} e^{\ad \uf_2}X = e^{\ad \uf_1} (X + \uf_2)\\
  & =\bigcup_{U\in \uf_1} e^{\ad U} (X + \uf_2)= \bigcup_{U\in
    \uf_1}(e^{\ad U}X + \uf_2) \\ &= \bigcup_{U\in \uf_1} (X
  +\lambda_1 U +\uf_2)= X +\uf\,.\tag*{\qedhere}
\end{align*}
\end{proof}

\par Continuing with the proof of Theorem \ref{LST}, we conclude that
the stabilizer $L\subseteq Q$ of $X_0\in\qf$ is a reductive
Levi-subgroup. Let
$$S:= \mu^{-1} (X_0)=\{z\in Z_0\mid \mu(z)=X_0\},$$ 
then for $q\in Q$ we have
\begin{equation}\label{L on S}
  qz_0\in S \Leftrightarrow \mu(qz_0)=X_0  \Leftrightarrow q X_0=X_0  \Leftrightarrow q\in L.
\end{equation}
Hence $L$ acts transitively on $S$. As $\mu$ is submersive, $S$ is a
submanifold of $Z_0$ and we obtain with
\begin{equation}\label{S-diffeo}
  Q\times_{L} S \to Z_0\, .
\end{equation}
a $Q$-equivariant diffeomorphism. As $L$-homogeneous space $S$ is
isomorphic to $L/L\cap H$.  Hence (1) is valid.

Note that (\ref{L on S}) implies that $(L\cap P)H=S\cap (PH)$, which
is open in $S$.  Thus $L/L\cap H$ is a real spherical space.

If (3) is valid, we are done. Otherwise we let $Q_1=Q$ and consider
the real spherical space $Z_1=L_1/L_1\cap H$ for $L_1=L$.  Iterating
the procedure of before yields a proper parabolic subgroup $R$ of
$L_1$ containing $L_1\cap P$ and with a Levi subgroup $L_2\subseteq
L_1$ such that
\begin{equation}\label{property of R}
  (R\cap N)\times L_2/(L_2\cap H) \to R\cdot z_0
\end{equation}
is a diffeomorphism.  We let $Q_2=RP=RU_1$, which is a subgroup since
$R$ normalizes $U_1$. Note that (\ref{property of R}) together with
the property (1) for $Q_1$ implies that this property is valid also
for $Q_2$.  We continue iterations until $H$ contains the non-compact
semisimple part of some $L_i$.  This will happen eventually since the
non-compact semisimple part of a Levi subgroup of $P$ is trivial.
\end{proof}

\subsection{$Z$-adapted parabolics}

\begin{definition}\label{defi Z-adapted} 
  Let $Z=G/H$ be a real spherical space.  A parabolic subgroup $Q<G$
  will be called {\it $Z$-adapted} provided that
  \begin{enumerate}
  \item There is a minimal parabolic subgroup $P\subseteq Q$ with $PH$
    open.
  \item There is a Levi decomposition $Q=LU$ such that $Q\cap
    H\subseteq L$.
  \item $\lf_{\mathrm n} \subseteq \hf$.
  \end{enumerate}
  A parabolic subalgebra $\qf$ of $\gf$ is called $Z$-adapted if it is
  the Lie algebra of a $Z$-adapted parabolic subgroup $Q$.
\end{definition}

\begin{theorem}\label{unique}
  Let $Z=G/H$ be an almost algebraic real spherical space and $P$ a
  minimal parabolic subgroup such that $PH$ is open.  Then there
  exists a unique parabolic subgroup $Q\supseteq P$ with unipotent
  radical $U$ such that $\uf$ is complementary to $\nf\cap \hf$ in
  $\nf$.  Moreover, this parabolic subgroup $Q$ is $Z$-adapted, and it
  is the unique parabolic subgroup above $P$ with that property.
\end{theorem}

\begin{proof} 
  Note first that if $Q\supseteq P$ and $Q=LU$ is a Levi decomposition
  then $\nf=(\nf\cap\lf)\oplus\uf$.  Assuming in addition (2) and (3)
  above then $\nf\cap\hf=\nf\cap\lf$, and hence $\nf\cap\hf$ is
  complementary to $\uf$.  Hence every $Z$-adapted parabolic subgroup
  $Q\supseteq P$ has this property of complementarity. In particular,
  this holds then for the parabolic subgroup $Q$ constructed with
  Theorem \ref{LST}.

  It remains to prove that if $Q'\supseteq P$ is another parabolic for
  which the unipotent radical $\uf'$ is complementary to $\nf\cap\hf$,
  then $Q'=Q$.  Since $\lf_n\subset \hf$ we find
$$\uf'\cap \lf\subseteq \uf'\cap\hf= \{0\}.$$
The lemma below now implies $\uf\supseteq\uf'$.  But then $\uf=\uf'$
since both spaces are complementary to $\nf\cap\hf$, and hence $Q=Q'$.
\end{proof}

\begin{lemma} Let $\pf$ be a minimal parabolic subalgebra, and let
  $\qf, \qf'\supseteq\pf$ be parabolic subalgebras with unipotent
  radicals $\uf$, $\uf'$.  If there exists a Levi decomposition
  $\qf=\lf+\uf$ such that $\lf\cap \uf'=\{0\}$, then $\qf\subseteq
  \qf'$.
\end{lemma}

\begin{proof} This follows easily from the standard description of the
  parabolic subalgebras containing $\pf$ by sets of simple roots.
\end{proof}

\subsection{The real rank of $Z$}
Let $Q$ be $Z$-adapted, with Levi decomposition $Q=LU$ as in
Definition \ref{defi Z-adapted}. From the local structure theorem we
obtain an isomorphism
$$ Q \times_L L/L\cap H\to Q\cdot z_0=P\cdot z_0\,.$$ 
Recall that $\lf_{\mathrm n} \subseteq \hf$.  We decompose
\begin{equation}\label{first decomposition l}
  \lf= \zf(\lf)\oplus[\lf,\lf]=\zf(\lf) \oplus \lf_{\mathrm c} \oplus \lf_{\mathrm n},
\end{equation}
where $\lf_{\mathrm c}$ denotes the sum of all compact simple ideals
in $\lf$.  Note that $D=L/L_{\mathrm n}$ is a Lie group with the Lie
algebra $\df=\zf(\lf)+\lf_{\mathrm c}$, which is compact, and that
$$\lf\cap\hf = \cf \oplus \lf_{\mathrm n}$$
with $\cf=\df\cap\hf$. Let $C=(L\cap H)/L_{\mathrm n}\subseteq D$,
then $L/L\cap H=D/C$, and
\begin{equation}
  U \times D/C\to P\cdot z_0\end{equation}
is an isomorphism. 

Consider the refined version of (\ref{first decomposition l}),
\begin{equation}\label{second decomposition l}
  \lf= \zf(\lf)_{\mathrm n\mathrm p} \oplus \zf(\lf)_{\mathrm c\mathrm p} \oplus \lf_{\mathrm c} \oplus \lf_{\mathrm n}
\end{equation}
in which $\zf(\lf)_{\mathrm n\mathrm p}$ and $\zf(\lf)_{\mathrm
  c\mathrm p}$ denote the non-compact and compact parts of $\zf(\lf)$.
Let $L=K_LA_L(L\cap N)$ be an Iwasawa decomposition of $L$, and let
$G=KAN$ be an Iwasawa decomposition of $G$ which is compatible, that
is, $K\supseteq K_L$ and $A=A_L$. Then $\af=\zf(\lf)_{\mathrm n\mathrm
  p}\oplus (\af\cap\lf_{\mathrm n})$.

Let $\af_h\subset \zf(\lf)_{\mathrm n\mathrm p}$ be the image of $\cf$
under the projection $\lf\to\zf(\lf)_{\mathrm n\mathrm p}$ along
(\ref{second decomposition l}), and let $\af_Z$ be a subspace of
$\zf(\lf)_{\mathrm n\mathrm p}$, complementary to $\af_h$. Then
\begin{equation}\label{third decomposition}
  \af=\af_Z\oplus\af_h\oplus (\af\cap\lf_{\mathrm n})
\end{equation}
The number $\dim\af_Z$ will be called the {\it real rank} of $Z$ in
Section \ref{Rsv}, where we show (under an additional hypothesis) that
it is an invariant of $Z$ (it is independent of the choices of $P$ and
$L$). See Remark \ref{rank remark}.

\subsection{$HP$-factorizations of a semi-simple group}

Let $Z=G/H$ be real spherical. In general $G/P$ admits several
$H$-orbits. Here we investigate the simplest case where there is just
one orbit.

\begin{prop} \label{ONTO} Let $G$ be semi-simple.  Assume that $Z=G/H$
  is real spherical and that $\hf$ contains no non-zero ideal of
  $\gf$.  Then $HP=G$ if and only if $H$ is compact.
\end{prop}

\begin{proof} Assume that $HP=G$. Note that then $HgP=G$ for every
  $g\in G$ and hence
$$\hf+\Ad(g)(\pf)=\gf$$
for every $g\in G$.

\par We first reduce to the case where $H$ is reductive in $G$.
Otherwise there exists a non-zero ideal $\hf_u$ in $\hf$ which acts
unipotently on $\gf$.  By conjugating $P$ if necessary we may assume
that $\hf_u\subseteq \nf$. It then follows from $G=PH$ that
$\Ad(g)(\hf_u)\subseteq \nf$ for all $g\in G$, which is absurd.

\par Assume now that $H$ is reductive and let $H=K_H A_H N_H$ be an
Iwasawa decomposition.  Let $X\in \af_H$ be regular dominant with
respect to $\nf_H$, and let $\qf$ be the parabolic subalgebra of $\gf$
which is spanned by the non-negative eigenspaces of $\ad X$. It
follows that $\qf\cap \hf$ is a minimal parabolic subalgebra of $\hf$,
and that $\nf_H$ is contained in the unipotent part $\uf$ of $\qf$.
As $Q$ contains a conjugate of $P$ we have $\gf=\hf+\qf$ and hence
$\dim (\hf/(\qf\cap \hf))=\dim (\gf/\qf)$, from which we deduce that
$\nf_H=\uf$. From $\nf_H=\uf$ and $\gf=\hf+\qf$ we deduce that
$\gf=\hf+\lf$.  Let $\hf_n$ be the subalgebra of $\hf$ generated by
$\nf_H$ and its opposite $\bar\nf_H$ with respect to the Cartan
involution of $H$ associated with $H=K_H A_H N_H$.  Then $\hf_n$ is
$\lf$-invariant and an ideal in $\hf$.  With $\gf=\hf+\lf$ we now
infer that $\hf_n$ is an ideal in $\gf$, and hence it is zero.  It
follows that $H=K_HA_H$, where $A_H$ is central in $H$.  We may assume
$K_H\subseteq K$ and $A_H\subseteq A$.  Then $G=HP$ implies $K=K_HM$,
and hence $K$ centralizes $A_H$.  This is impossible unless
$A_H=\{1\}$ and then $H$ is compact.

Conversely, if $H$ is compact then the open $H$-orbit on $G/P$ is
closed, and since $G/P$ is connected it follows that $HP=G$.
\end{proof}

\section{Real spherical varieties}\label{Rsv}

All complex varieties $Z_\C$ in this section will be defined over
$\R$.  Typically we denote by $Z$ the set of real points of $Z_\C$. If
$Z$ is Zariski-dense in $Z_\C$, then we call $Z$ a {\it real
  (algebraic) variety}.

\par We say that a subset $U\subset Z$ is (quasi)affine if there
exists a (quasi)affine subset $U_\C\subset Z_\C$ such that $U=U_\C\cap
Z$.

\begin{rmk} Even if $Z_\C$ is irreducible it might happen that $Z$ has
  several connected components with respect to the Euclidean
  topology. However, by Whitney's theorem, the number of connected
  components is always finite.  Take for example $Z=\GL(n,\R)$ and
  $Z_\C=\GL(n,\C)$. Here $Z$ breaks into two connected components
  $\GL(n,\R)_+$ and $\GL(n,\R)_-$ characterized by the sign of the
  determinant; certainly it would be meaningful to call $\GL(n,\R)_+$
  a real algebraic variety as well.
  \par Let $Z_1\amalg \ldots \amalg Z_n$ be the decomposition of $Z$
  into connected components (with respect to the Euclidean topology).
  A more general notion of {\it real variety} would be to allow
  arbitrary unions of those $Z_j$ which are Zariski dense in
  $Z_\C$. In fact, all the statements derived in this section for real
  varieties are valid in this more general setup.
\end{rmk}

In this section we let $G$ be a real algebraic reductive group and
$G_\C\supseteq G$ its complexification. Furthermore, $P$ is a minimal
parabolic subgroup of $G$ and $P=MAN$ a Langlands decomposition of it.
\par By a {\it real $G$-variety $Z$} we understand a real variety $Z$
endowed with a real algebraic $G$-action. A real $G$-variety will be
called {\it linearizable} provided there is a finite dimensional real
$G$-module $V$ such that $Z$ is realized as real subvariety of
$\mathbb{P}(V)$.

An algebraic real reductive group $G$ is called {\it elementary} if
$G\cong M \times A$ with $M$ compact and $A=(\R^+)^l$.  This is
equivalent to $G=P$. A real $G$-variety $Z$ will then be called
\emph{elementary} if $G/J$ is elementary where $J$ is the kernel of
the action on $Z$.

\begin{definition}\label{defi spherical variety} 
  A linearizable real $G$-variety $Z$ will be called {\it real
    spherical} provided that:
  \begin{itemize}
  \item $Z_\C$ is irreducible,
  \item $Z$ admits an open $P$-orbit.
  \end{itemize}
\end{definition}

\begin{rmk}\label{rmk elementary} (a)  In the definition of a (complex)
  spherical variety one requests in particular that the variety is
  normal.  We now explain how this is related to our notion of real
  spherical.
  \par
  Assume that $Z_\C$ is normal. Then it follows from a theorem of
  Sumihiro (\cite{KKLV} p.~64) that every every point $z\in Z_\C$ has
  a $G_\C$-invariant open neighborhood $U$ which can be equivariantly
  embedded into $\mathbb{P}(V_\C)$ where $V_\C$ is a finite
  dimensional representation of $G_\C$. It follows that if $z\in Z$
  then $U_0:=(U\cap\overline{U})\cap Z$ is a linearizable open
  neighborhood of $z$. Observe that there is always a normalization
  map $\nu:\tilde Z\to Z$ where $\tilde Z$ is a normal $G$-variety and
  $\nu$ is proper, finite to one, and invertible over an open dense
  subset of $Z$.
  \par (b) If $Z$ is a real spherical variety, then the number of open
  $P$-orbits is finite: As $Z_\C$ is irreducible, there is exactly one
  open $P_\C$-orbit on $Z_\C$ and the real points of this open
  $P_\C$-orbit decomposes into finitely many $P$-orbits. We conclude
  in particular that there are only finitely many open $G$-orbits in
  $Z$. Let ${\mathcal O}\simeq G/H$ be one of them.  Then $G/H$ is a
  real spherical algebraic homogeneous space which we considered
  before.
  \par (c) Let $Z$ be an elementary real spherical variety. If $G=A$,
  then $Z$ consists of the real points of a toric variety defined over
  $\R$.
  \par (d) Let $G=M\times A$ be an elementary algebraic real reductive
  group and $Z=G/H$ a homogeneous real spherical $G$-variety.  Since
  there are no algebraic homomorphisms between a split torus and a
  compact group, the group $H$ splits as $H=M_0\times A_0$ with
  $M_0\subseteq M$ and $A_0\subseteq A$. Thus $Z=M/M_0\times A/A_0$.
\end{rmk}

\subsection{Some general facts about real $G$-varieties}

Let $Z$ be an irreducible real variety. We denote by $\C[Z]$,
resp.~$\C(Z)$, the ring of {\it regular}, resp.~{\it rational
  functions} on $Z$, that is $\C[Z]$ consists of the restrictions of
the regular functions on $Z_\C$ to $Z$ and likewise for $\C(Z)$.

\par As $Z$ is Zariski-dense we observe that the restriction mapping
$\operatorname{Res}: \C(Z_\C)\to \C(Z)$ is bijective. Next we note
that both $\C(Z)$ and $\C[Z]$ are invariant under complex conjugation
$f\mapsto \oline f$. In particular with $f\in \C[Z]$, resp.~$\C(Z)$,
we also have that $\operatorname{Re} f$ and $\operatorname{Im} f$
belong to $\C[Z]$, resp.~$\C(Z)$.

\par If a compact real algebraic group $M$ acts on $Z$, then the
$M$-average
$$f\mapsto f^M;   \ \ f^M(z):= \int_M f(m\cdot z) \ dm \qquad (z\in Z)$$
preserves $\C[Z]$. This follows from the fact that the $G$-action on
$\C[Z]$ is locally finite. Put together we conclude
\begin{equation}\label{average}
  f\in \C[Z]\Rightarrow  (|f|^2)^M \in \C[Z]^M\text{ with } f\ne0\Rightarrow  (|f|^2)^M \ne0.\end{equation}  
   
\par Let us denote by $\hat P$ the set of real algebraic characters
$\chi: P \to \R^\times$ such that $MN \subseteq \ker \chi$. Note that
the subgroup $MN$ of $P$, and hence $\hat P$, is independent of the
choice of a Langlands decomposition of $P$.  However, when that has
been chosen, there is a natural identification of $\hat P$ with a
lattice $\Lambda\subseteq \af^*$.

\par For the rest of this subsection we let $Z$ be a real
$G$-variety. We denote by $\C(Z)^{(P)}$ the set of $P$-semi-invariant
functions, i.e.~the rational functions $f\in \C(Z)\setminus\{0\}$ for
which there is a $\chi\in \hat P$ such that $f(p^{-1}z) =\chi(p) f(z)$
for all $p\in P$, $z\in Z$ for which both sides are defined.  We
denote by $\C(Z)^P$ the set of $P$-invariants in $\C(Z)$.  Likewise we
define $\C[Z]^{(P)}$ and $\C[Z]^P$.  Further we denote by $\R(Z)$ and
$\R[Z]$ the real valued functions in $\C(Z)$ and $\C[Z]$.

\begin{lemma} \label{quotient} Let $Z$ be a quasi-affine real
  $G$-variety. Then for all non-zero $f\in \R(Z)^P$ there exists $f_1,
  f_2 \in \R[Z]^{(P)}$ such that $f=\frac{f_1}{f_2}$.
\end{lemma}

\begin{proof} Let $f\in \R(Z)^{P}$. As $Z$ is quasi-affine, we find
  regular functions $h_1, h_2\in \C[Z]$, $h_2\ne0$ such that
  $f=\frac{h_1}{h_2}$. Consider the ideal
$$I:=\{ h\in \C[Z]\mid h f \in \C[Z]\}\, .$$
Note that:
\begin{itemize}
\item $I\neq \{0\}$ as $h_2\in I$,
\item $I=\oline I$ as $f$ is real,
\item $I$ is $P$-invariant as $f$ is $P$-fixed.
\end{itemize}
The action of $P$ on $\C[Z]$ is algebraic, hence locally finite and
thus we find an element $0\ne h\in I$ which is an eigenvector for the
solvable group $AN$.  We use (\ref{average}) to obtain with
$f_2=(|h|^2)^M$ a non-zero element of $I\cap \R[Z]^{(P)}$. Now we put
$f_1=f_2f\in\R[Z]^{(P)}$.
\end{proof}

For $\chi \in \hat P=\Lambda$ we let
$$\C[Z]_\chi:=\{ f\in \C[Z] \mid (\forall p \in P, z\in Z)\ f(p^{-1}z)=\chi(p) f(z)\}\, $$ 
and likewise define $\C(Z)_\chi$. We define a sub-lattice of $\Lambda$
by
$$\Lambda_Z:=\{ \chi \in \hat P\mid \C(Z)_\chi\neq \{0\}\}\, .$$
With that we declare the {\it real rank} of $Z$ by
\begin{equation}\label{rank}
  \rk_\R (Z):= \dim_\Q  (\Lambda_Z\otimes_\Z \Q)\, .
\end{equation}
It is easily seen that $\rk_\R(Z)$ is independent of the choice of
minimal parabolic subgroup $P$.

\begin{rmk}\label{rank remark}
  Let $Z=G/H$ be homogeneous. Then $\rk_\R(Z)=\dim \af_Z$ where
  $\af_Z$ is defined by (\ref{third decomposition}). In fact, as a
  $Q$-variety, an open subset of $Z$ is isomorphic to $U\times L/L\cap
  H$. Thus $\R(Z)^{(P)}=\R(L/L\cap H)^{(L\cap P)}$.  Since $H$
  contains $L_{\mathrm n}$ the variety $L/L\cap H$ is elementary.  By
  Remark \ref{rmk elementary}(d), we have $\R(L/L\cap H)^{(L\cap
    P)}=\R(A/A_0)^{(A)}$ which implies the claim, as
  $A/A_0\simeq\af_Z$.
\end{rmk}

\begin{lemma}\label{affine} Let $Z$ be a linearizable irreducible real $G$-variety and 
  $Y\subseteq X$ a Zariski-closed $G$-invariant subvariety. Then there
  exists a $P$-stable affine open subset $Z_0\subseteq Z$ which meets
  $Y$ and such that the restriction mapping:
$$\R[Z_0]^{(P)} \to \R[Z_0\cap Y]^{(P)}$$
is onto.
\end{lemma}

\begin{proof} If $G$ is complex, then this is the real points version
  of \cite{Brion}, Prop. 1.1. Further with $P$ replaced by $AN$ one
  can literally copy the proof of \cite{Brion}. Finally the additional
  $M$-invariance when moving from $AN$ to $P$ is obtained from
  (\ref{average}).
\end{proof}

\par Denote by $\Lambda^+\subseteq \Lambda$ the semigroup of elements
dominant with respect to $P$.  For all $\lambda \in \Lambda^+$ we set
$$m(\lambda):= \dim_\C  \C[Z]_\lambda\, .$$
If we identify $\Lambda^+$ with a subset of the irreducible finite
dimensional representations of $G$, then $m(\lambda)$ is the
multiplicity of the irreducible representation $\lambda$ occurring in
the locally finite $G$-module $\C[Z]$. The following is a real
analogue of the Vinberg-Kimel'feld theorem \cite{VinKim}:

\begin{prop}\label{m=1} 
  Let $Z$ be a quasi-affine irreducible $G$-variety. Then the
  following assertions are equivalent:
  \begin{enumerate}
  \item $Z$ is real spherical.
  \item $m(\lambda)\leq 1 $ for all $\lambda \in \Lambda^+$.
  \end{enumerate}
\end{prop}

\begin{proof} ``$(1)\Rightarrow (2)$'': Let $z\in Z$ such that $P\cdot
  z$ is open in $Z$. Then two $P$-semi-invariant functions $f_1$ and
  $f_2$ with respect to the same character $\lambda\in \hat P$ satisfy
  $f_1 |_{P\cdot z} = c f_2|_{P\cdot z}$ for some constant $c\in
  \C$. As $Z_\C$ is irreducible we conclude that $f_1 = c f_2$.
  \par ``$(2)\Rightarrow (1)$'': We recall that there is an open
  $P$-orbit on $Z$ if and only if $\C(Z)^P=\C\1 $. This follows from
  Rosenlicht's theorem, \cite{Spr} p.~23, applied to $Z_\C$.  Let now
  $f\in \C(Z)^P$. According to Lemma \ref{quotient} there exists $f_1,
  f_2 \in \C[Z]^{(P)}$ such that $f=\frac{f_1}{f_2}$.  Clearly $f_1$
  and $f_2$ correspond to the same character $\lambda\in \hat P$. As
  $m(\lambda)\leq 1$, we conclude that $f_1$ is a multiple of $f_2$.
\end{proof}

\begin{cor}\label{cor1} Let $Z$ be a real spherical variety and $Y\subseteq Z$ a closed $G$-invariant irreducible 
  subvariety. Then $Y$ is real spherical.
\end{cor}

\begin{proof} If $Z$ is quasi-affine, then this is immediate from the
  previous proposition as the restriction mapping $\C[Z]\to \C[Y]$ is
  onto.  The more general case is reduced to that by considering the
  affine cone over $Z$. Recall that $Z\subseteq\mathbb{P}(V)$. The
  preimage of $Z$ in $V\bs \{0\}$ will be denoted by $\hat Z$. Note
  that $\hat Z$ is quasi-affine. Moreover $Z$ is real spherical if and
  only if $\hat Z$ is real spherical for the enlarged reductive group
  $G_1 = G \times \R^\times$.
\end{proof}

\begin{cor}\label{cor2} Let $Z$ be a real spherical variety. Then the number of $G$-orbits on $Z$ is finite
  and each $G$-orbit is spherical.
\end{cor}

\begin{proof} In view of the preceding corollary we only need to show
  that there are finitely many $G$-orbits. Suppose that there are
  infinitely many $G$-orbits. We let $Y\subseteq Z$ be a closed
  irreducible $G$-subvariety of minimal dimension which admits
  infinitely many $G$-orbits. By Corollary \ref{cor1}, $Y$ is
  spherical.  In particular $Y$ admits open $G$-orbits. After deleting
  the finitely many open $G$-orbits from $Y$ we obtain a $G$-invariant
  subvariety $Y_1\subseteq Y$ with infinitely many $G$-orbits.  As
  $\dim Y_1 < \dim Y$ we reach a contradiction.
\end{proof}

The main result of \cite{KS} was that every homogeneous real spherical
space admits only finitely many $P$-orbits. With Corollary \ref{cor2}
we then conclude:

\begin{theorem} Let $Z$ be a real spherical variety. Then the number
  of $P$-orbits on $Z$ is finite.
\end{theorem}

\subsection{The local structure theorem}

\par Let $Z$ be a real spherical variety and $Y\subseteq Z$ a
$G$-invariant closed subvariety. Our goal is to find a $P$-invariant
coordinate chart $Z_0$ for $Z$ which meets $Y$.  For that we may
assume that $Z$ is Zariski-closed in $\mathbb{P}(V)$, where $V$ is a
finite-dimensional $G$-module.  Moreover we may assume that
$Y\subseteq Z$ is a closed $G$-orbit. In particular $Y$ is real
spherical by Corollary \ref{cor2} and we let $Q_Y<G$ be a $Y$-adapted
parabolic.

Under these assumption on $Y$ and $Z$ there is the following immediate
generalization of Lemma \ref{affine}.

\begin{lemma}\label{affine2} Let $Z$ be real spherical variety, 
  closed in $\mathbb{P}(V)$, and $Y\subseteq Z$ a closed
  $G$-orbit. Then there exists a $Q_Y$-stable affine open subset
  $Z_0\subseteq Z$ which meets $Y$ and such that the restriction
  mapping:
$$\R[Z_0]^{(Q_Y)} \to \R[Z_0\cap Y]^{(Q_Y)}$$
is onto.
\end{lemma}

\begin{proof} The proof is analogous to the one of Lemma \ref{affine}.
  We obtain $Z_0$ is the non-vanishing locus of a $Q_Y$-semi-invariant
  homogeneous polynomial function on $V$.
\end{proof}

\begin{cor}\label{affineopen} Let $Z\subset \mathbb{P}(V)$ be a closed 
  real spherical variety and $Y$ an elementary closed subvariety. Then
  there exists a $G$-stable affine open subset $Z_0\subset Z$ such
  that $Z_0\cap Y\neq \emptyset$.
\end{cor}
\begin{proof} One has $Q_Y=G$. \end{proof}

We now start with the construction of $Z_0$. In case $Y$ is
elementary, $Z_0$ is given by Corollary \ref{affineopen}. So let us
assume that $Y$ is not elementary, i.e. $G_{\mathrm n}$ does not act
trivially on $Y$.  Let $\bar P=MA\bar N$ be opposite to $P$.  As
$Y\subseteq \mathbb{P}(V)$ is closed, we can find a vector $y_0\in V$
such that $[y_0]\in Y$ is $A\bar N$-fixed, and such that $A$ acts by a
non-trivial character on $y_0$.  This can be seen as follows. Assume
for simplicity that $V$ is irreducible.  Then $Y$ contains a vector
$y$ of which the $A$-weight decomposition has a non-trivial component
$y_0$ in the lowest weight space of $V$.  Compression of $y$ by $A^+$
then exhibits a non-zero multiple of $y_0$ as a limit of elements from
$Y$.

Next we choose $v_0^*\in V^*$ such that $[v_0^*]$ is $AN$-fixed and
$v_0^*(y_0)=1$.  Let $\chi: A \to \R^+$ be the character defined by
$a\cdot v_0^* = \chi(a)v_0^*$.

Consider the function
$$F: V \to \R,\ \  v\mapsto \int_M v_0^*(m\cdot v)^2 \ dm$$
and note that
$$F(man\cdot v)= \psi(a) F(v)$$
for all $man\in MAN$ and $v\in V$, where $\psi=\chi^{-2}$. Further $F$
is real algebraic and homogeneous of degree $2$. Thus $\{ [v]\in
\mathbb{P}(V)\mid F(v)\neq 0\}$ defines an affine open set in
$\mathbb{P}(V)$ and the intersection with $Z$ yields an affine open
set $Z_0$.  Note that $F$ is not constant and hence $Z_0$ is a proper
subvariety.  We define $Q\supseteq P$ to be the parabolic subgroup
which fixes the line $\R F|_{Z_0}$, that is $Q=\{ g\in G\mid
gZ_0=Z_0\}$.

As before we define on $Z_0$ a moment-type map:
$$\mu: Z_0\to \gf^*, \ \ \mu(z)(X):=\frac{dF(v)(X)}{F(v)}\, $$
for $z=[v]\in Z\subseteq \mathbb{P}(V)$. This map is algebraic and
$Q$-equivariant. Let $U<Q$ be the unipotent radical.

We claim that $\im \mu$ is a $Q$-orbit. In fact for $X\in \qf$ we have
$\mu(z)(X)=d\psi(X)$ for all $z\in Z$, and after identifying $\gf$
with $\gf^*$ we obtain as in the previous section that
$$\im \mu = \Ad(Q) X_0= X_0 + \uf $$
with $X_0=\mu([y_0])$. The stabilizer of $X_0$ determines a
Levi-subgroup $L<Q$.  Then $S:=\mu^{-1}(X_0)$ is an $L$-stable affine
subvariety of $Z_0$ and we obtain an algebraic isomorphism
$$ Q \times_L S \to Z_0\, .$$ 
The affine $L$-variety $S$ is real spherical and meets $Y$.  We
continue the procedure with $(L,S,S\cap Y)$ instead of $(G,Z,Y)$.  The
procedure will stop at the moment when $S\cap Y$ is fixed under
$L_{\mathrm n}$.  We have thus shown:

\begin{theorem}[Local structure theorem, general case] Let $Z$ be a
  real spherical variety and $Y\subseteq Z$ a closed $G$-invariant
  subvariety. Then there is parabolic subgroup $Q\supseteq P$ with
  Levi-decomposition $Q=LU$ with the following properties: There is a
  $Q$-invariant affine open piece $Z_0\subseteq Z$ meeting $Y$ and an
  $L$-invariant closed spherical subvariety $S\subseteq Z_0$ such
  that:
  \begin{enumerate}
  \item There is a $Q$-equivariant isomorphism
$$ Q \times_L S \to Z_0\, .$$ 
\item $S\cap Y$ is an elementary spherical $L$-variety.
\end{enumerate}
\end{theorem}

\section{The normalizer of a spherical subalgebra}

As in the preceding section we assume that $G$ is algebraic and let
$\hf$ be the Lie algebra of a spherical subgroup $H<G$.  We denote by
$\tilde \hf:=\nf_\gf(\hf)$ the normalizer of $\hf$ in $\gf$ and by
$\tilde H$ the normalizer in $G$.  Note that $\hf \triangleleft
\tilde\hf$ is an ideal.  Let $ \pf$ be a minimal parabolic subalgebra
such that $\pf +\hf =\gf$ and let $\qf$ denote the unique parabolic
subalgebra above $\pf$, which is $Z$-adapted. Let $\tilde Z=G/\tilde
H$.

\begin{lemma}\label{same Q}
  The parabolic subalgebra $\qf$ is also $\tilde Z$-adapted.
\end{lemma}

\begin{proof}
  We write $\tilde \qf$ for the unique $\tilde Z$-adapted parabolic
  above $\pf$ and $\tilde\uf$ for its unipotent radical. Then
$$\nf=(\nf\cap\hf)\oplus\uf= (\nf\cap\tilde\hf)\oplus\tilde\uf$$
It follows that $\tilde\uf\subseteq\uf$ and $\qf\subseteq\tilde\qf$.
To obtain a contradiction we assume that $\qf \subsetneq
\tilde\qf$. Then $\tilde\uf\subsetneq\uf$ and $\nf\cap\hf\subsetneq
\nf\cap\tilde\hf$.  In particular, the Lie algebra $\tilde\hf/ \hf$
cannot be compact.

To conclude the proof we now show that $\tilde\hf/\hf$ is compact.
Suppose first that $Z$ is quasi-affine and let $\C[Z]= \bigoplus_{\pi
  \in \hat G} \C[Z]_\pi$ be the decomposition of the $G$-module
$\C[Z]$ into $G$-isotypical components.  For each $\pi$ we choose a
model space $V_\pi$ and let $\M_\pi := \operatorname{Hom}_G (V_\pi ,
\C[Z])$ be the corresponding multiplicity space. Note that $\M_\pi$ is
finite dimensional as there is a natural identification of $\M_\pi$
with the space of $H$-fixed elements in $V_\pi^*$.

Let $C:=\tilde H/ H$. Note that $C$ acts from the right on $\C[Z]$ and
preserves each $\C[Z]_\pi$, thus inducing an action on $\M_\pi$. Since
$Z$ is quasi-affine we can choose finitely many $\pi_1, \ldots, \pi_k$
so that we obtain a faithful representation of $C$ on the sum
$\M:=\bigoplus_{j=1}^k \M_{\pi_j}$.

\par Let $B<G_\C$ be a Borel subgroup contained in $P_\C$.  For every
$\pi$ we let $v_\pi$ be a $B$-highest weight vector in $V_\pi$.  To
every $\eta\in \M_\pi$ we associate the function $f_\eta(g)=
\eta(\pi(g^{-1})v_\pi)$ and define an inner product on $\M_\pi$ by
$$\la \eta, \eta\ra_\pi :=  (|f_\eta|^2)^M (z_0)\, $$
with notation of (\ref{average}).  As $(|f_\eta|^2)^M$ is a matrix
coefficient of a representation in $\Lambda$, and as multiplicities
for these are at most one by Proposition \ref{m=1}, we obtain that
there is a real character $\chi_\pi : C \to \R^\times$ such that
$$ \la  h \cdot \eta, h\cdot \eta\ra_\pi= \chi_\pi(h)\la \eta, \eta\ra_\pi\, .$$ 

The group $C_1:=\bigcap_{j=1}^k \ker \chi_{\pi_j}$ acts unitarily and
faithfully on $\M$, hence is compact.  By definition
$C/C_1<(\R^\times)^k$, hence the Lie algebra of $C$ is compact.

Finally we reduce to the quasi-affine case using the affine cone over
$\mathbb{P}(V)$ as before, see the proof of Corollary \ref{cor1}.
\end{proof}

Let $Q=LU$ be a Levi decomposition as in Definition \ref{defi
  Z-adapted} and recall the decomposition (\ref{second decomposition
  l}).

\begin{prop} The normalizer $\tilde \hf$ of $\hf$ is of the form
  \begin{equation}\label{direct sum}
    \tilde \hf = \hf \oplus \tilde \cf
  \end{equation}
  with $\tilde\cf$ a subalgebra of the form $\tilde \cf = \tilde \af
  \oplus \tilde \mf$ where $\tilde\af< \zf(\lf)_{\mathrm n\mathrm p}$
  and $\tilde \mf< \zf(\lf)_{\mathrm c\mathrm p} +\lf_{\mathrm c}$.
\end{prop}

\begin{proof} From Lemma \ref{same Q} we conclude that $\tilde\hf =
  \hf + \tilde\hf \cap \lf$, and we obtain (\ref{direct sum}) with a
  subspace $\tilde\cf$ of $\zf(\lf) +\lf_{\mathrm c}$. It is a
  subalgebra because $\zf(\lf) +\lf_{\mathrm c}$ is reductive and
  $\hf$ is an ideal in $\tilde\hf$.
  \par Write $\tilde\af$ for the orthogonal projection of $\tilde\cf$
  to $\zf(\lf)_{\mathrm n\mathrm p}$ and $\tilde\mf $ for the
  orthogonal projection of $\tilde\cf$ to $\zf(\lf)_{\mathrm c\mathrm
    p} +\lf_{\mathrm c}$. Then, $\tilde\cf \subseteq \tilde \af
  +\tilde \mf$ and it remains to show equality. This will follow if we
  can show that both $\tilde\af$ and $\tilde\mf$ normalize $\hf$. For
  that we decompose $X\in \tilde\cf$ as $X=X_a+X_m$ with $X_a\in
  \tilde \af$ and $X_m\in \tilde\mf$.  Observe that $\ad X_a$ commutes
  with $\ad X_m$. Both operators are diagonalizable with real,
  resp. imaginary, spectrum.  As $\ad X$ preserves $\hf$ we therefore
  conclude that $\ad X_a$ and $\ad X_m$ preserve $\hf$ as well.
\end{proof}

\begin{cor} 
  Let $H\subseteq G$ be real spherical. Then $N_G(H)/H$ is an
  elementary group.
\end{cor}

\begin{cor} The normalizer $\tilde\hf$ is its own normalizer:
  $\tilde{\tilde\hf}=\tilde{\hf}$.
\end{cor}

\begin{proof} It suffices to show that the normalizer
  $\tilde{\tilde\hf}$ of $\tilde\hf$ normalizes $\hf$, as well.  Let
  $\tilde H=N_G(\hf)$.  Observe that $\tilde H/H$ is an elementary
  real algebraic group, it is in particular reductive. Thus,
  $\tilde\hf_u=\hf_u$ for the nilpotent radicals. This implies that
  $\tilde{\tilde\hf}$ normalizes $\hf_u$ and that $\tilde H/H_u$ is a
  reductive real algebraic group. A connected group, which acts by
  algebraic automorphisms on a reductive Lie group, acts by inner
  automorphisms, hence fixes every ideal. Thus
  $\hf/\hf_u\subseteq\tilde\hf/\hf_u$ is normalized by
  $\tilde{\tilde\hf}$, as well.
\end{proof}

\begin{rmk}
  On the group level, the statement is wrong. Let, e.g., $G=GL(2,\R)$
  and $H=\left(\begin{array}{cc}*&0\\0&1\end{array}\right)$. Then
  $N_G(H)=T=\left(\begin{array}{cc}*&0\\0&*\end{array}\right)$. Thus
  $N_G(N_G(H))=N_G(T)$ is strictly larger than $N_G(H)=T$.

\end{rmk}


\begin{thebibliography} {10}

\bibitem{Borel} A.~Borel, {\it Linear Algebraic Groups},
  Springer-Verlag 1991.

\bibitem{Brion} M.~Brion, {\it Vari\'et\'es Sph\'eriques}, Notes de la
  session de la SMF, ``Op\'erations hamiltoniennes et op\'erations de
  groupes alg\'ebriques'' (Grenoble, 1997), {\tt
    www-fourier.ujf-grenoble.fr/\lower2pt\hbox{\string~}mbrion/spheriques.pdf}

\bibitem{BLV} M.~Brion, D.~Luna and T.~Vust, {\it Espaces homog\'enes
    sph\'eriques}, Invent.~Math.~{\bf 84} (1986), 617--632.

\bibitem{Kn} F.~Knop, {\it The asymptotic behavior of invariant
    collective motion}, Invent.~Math.~{\bf 116} (1994), 309--328.

\bibitem{KKLV} F.~Knop, H.~Kraft, D.~Luna and T.~Vust, {\it Local
    properties of algebraic group actions}. Algebraische
  Transformationsgruppen und Invariantentheorie, 63--75, DMV Sem.~{\bf
    13}, Birkh\"auser, Basel, 1989.

\bibitem{Kraemer} M. Kr\"amer, {\it Sph\"arische Untergruppen in
    kompakten zusammenh\"angenden Gruppen}, Compositio Math. {\bf 38}
  (1979), 129--153.

\bibitem{KS} B.~Kr\"otz and H.~Schlichtkrull, {\it Finite orbit
    decomposition of real flag manifolds}, arXiv:1307.2375. To appear
  in J. Eur. Math. Soc. (JEMS)

\bibitem{Spr} T.~A.~Springer, {\it Aktionen reduktiver Gruppen auf
    Variet\"aten}. Algebra\-ische Transformationsgruppen und
  Invariantentheorie, 3--39, DMV Sem.~{\bf 13}, Birkh\"auser, Basel,
  1989.

\bibitem{VinKim} E.~Vinberg and B.~Kimel'feld, {\it Homogeneous
    domains on flag manifolds and spherical subsets of semisimple Lie
    groups}, Funktsional. Anal. i Prilozhen. {\bf 12} (1978), 12--19.

\end{thebibliography}
\end{document}